\documentclass[12pt]{amsart}

\usepackage{amssymb}

\usepackage{enumitem}

\usepackage{graphicx}

\makeatletter
\@namedef{subjclassname@2020}{%
  \textup{2020} Mathematics Subject Classification}
\makeatother

\usepackage[T1]{fontenc}
\newtheorem{theorem}{Theorem}[section]
\newtheorem{lemma}[theorem]{Lemma}
\newtheorem{proposition}[theorem]{Proposition}
\newtheorem{corollary}[theorem]{Corollary}

\newtheorem{mainthm}{Main Theorem}

\numberwithin{equation}{section}

\newcommand{\cC}{\mathcal C}
\newcommand{\cK}{\mathcal K}
\newcommand{\dH}{d_{\mathrm H}}
\newcommand{\dist}{\operatorname{dist}}
\newcommand{\co}{\overline{\operatorname{co}}}
\newcommand{\norm}[1]{\left\lVert #1\right\rVert}
\newcommand{\abs}[1]{\left\lvert #1\right\rvert}
\newcommand{\pair}[2]{\left\langle #1,#2\right\rangle}

\begin{document}

\title{Surjective Hausdorff isometries of hyperspaces of bounded closed convex sets}

%%% Give just the affiliation and not the detailed postal address.

\author[L. Cheng]{Lixin Cheng}
\address{School of Mathematical Sciences\\
Xiamen University\\
361005 Xiamen, China}
\email{lxcheng@xmu.edu.cn}

\author[W. He]{Wuyi He}
\address{School of Mathematics and Statistics\\
Chongqing University of Posts and Telecommunications\\
400065 Chongqing, China}
\email{hewy@cqupt.edu.cn}

\author[C. Liu]{Chulei Liu}
\address{School of Mathematics and Statistics\\
Fuzhou University\\
350108 Fuzhou, China}
\email{chuleiliu@fzu.edu.cn}

\author[Z. Zheng]{Zheming Zheng}
\address{School of Mathematics and Statistics\\
Jiangxi Normal University\\
330022 Nanchang, China}
\email{proofshine@jxnu.edu.cn}

\date{}

\begin{abstract}
Let $X$ and $Y$ be real Banach spaces, and let $\cC(X)$ and $\cC(Y)$ denote the families of all nonempty bounded closed convex subsets of $X$ and $Y$, respectively, equipped with the Hausdorff metric. We prove that every surjective isometry $F:\cC(X)\to\cC(Y)$ is induced by a surjective affine isometry of the underlying spaces. More precisely, $F$ maps singleton sets onto singleton sets, and the map $T:X\to Y$ defined by $F(\{x\})=\{T(x)\}$ is a surjective affine isometry satisfying
\[
F(A)=T[A]:=\{T(x):\;x\in A\},\qquad(A\in\cC(X)).
\]
In particular, our theorem extends the result of Gruber and Lettl for finite-dimensional Euclidean spaces to arbitrary real Banach spaces, without imposing any additional assumptions on the underlying spaces.
\end{abstract}

%% Only one classification should be declared as primary.



\maketitle

\section{Introduction}
The celebrated Mazur--Ulam theorem \cite{MazurUlam1932} asserts that every surjective isometry between two real normed spaces is necessarily affine. Since its appearance, a substantial part of the theory has been devoted to weakening either the surjectivity assumption or the exact preservation of distances. Nonsurjective isometries were studied, for example, by Figiel \cite{Figiel1968} and Godefroy and Kalton \cite{GodefroyKalton2003}. Approximate isometries and their surjective variants were investigated by Hyers and Ulam \cite{HyersUlam1945,HyersUlam1947}, Gevirtz \cite{Gevirtz1983}, Gruber \cite{Gruber1978}, Omladi\v{c} and \v{S}emrl \cite{OmladicSemrl1995}, Vestfrid \cite{Vestfrid2015}, and Cheng, Shen, Zhang and Zhou \cite{ChengShenZhangZhou2017}. The stability and weak stability of general or nonsurjective $\varepsilon$-isometries were further developed by Qian \cite{Qian1995}, \v{S}emrl and V\"ais\"al\"a \cite{SemrlVaisala2003}, Cheng, Dong and Zhang \cite{ChengDongZhang2013}, Cheng and Zhou \cite{ChengZhou2014}, Cheng, Cheng, Tu and Zhang \cite{ChengChengTuZhang2015}, Cheng, Tu and Zhang \cite{ChengTuZhang2016}, and Cheng and Dong \cite{ChengDong2020}. Coarse isometries provide another natural weakening: surjective coarse isometries were considered by Lindenstrauss and Szankowski \cite{LindenstraussSzankowski1985}, while the nonsurjective case was studied by Cheng, Fang, Luo and Sun \cite{ChengFangLuoSun2019}. These results show that the rigidity phenomenon behind the Mazur--Ulam theorem persists, in various forms, well beyond surjective exact isometries.

The set-valued representation problem was first investigated in finite dimensions. In 1975, Schneider \cite{Schneider1975} studied surjective Hausdorff isometries of the hyperspace of nonempty compact convex subsets of Euclidean space and proved that such isometries are induced by isometries of the underlying Euclidean space. In 1980, Gruber and Lettl \cite{GruberLettl1980} strengthened Schneider's result by showing that the surjectivity assumption on the hyperspace isometry can be omitted. In 1986, Bandt \cite{Bandt1986} obtained a further finite-dimensional extension: the corresponding representation remains valid for the hyperspace of all nonempty compact subsets, not necessarily convex, of a convex body in a finite-dimensional strictly convex normed space. Thus, the early theory established a strong finite-dimensional rigidity phenomenon for Hausdorff isometries of natural hyperspaces.

The first infinite-dimensional result in this direction was obtained by Zhou, Zhang and Liu in 2018 \cite{ZhouZhangLiu2018}. They considered hyperspaces of nonempty compact convex subsets of Banach spaces and proved the desired representation under a weak-star smoothness assumption on the dual unit balls, together with the normalization that the isometry fixes the origin. More precisely, under these assumptions a surjective Hausdorff isometry induces a surjective linear isometry between the underlying Banach spaces and acts pointwise on every compact convex set. They also explicitly asked whether the conclusion remains valid for arbitrary Banach spaces and whether the condition that the origin be fixed can be removed.

Cheng and Zheng subsequently developed the problem in two related directions. For the hyperspace of all nonempty bounded closed convex subsets, they proved in \cite{ChengZheng2021} that if at least one of the underlying Banach spaces is Asplund, then every surjective Hausdorff isometry preserves singleton sets and therefore induces a surjective affine isometry between the underlying spaces. Under the additional assumption that one of the spaces is Fr\'echet smooth or locally uniformly convex, they obtained the full setwise representation. For the compact-convex hyperspace, Cheng and Zheng later removed the Asplund assumption completely \cite{ChengZheng2022}: for arbitrary real Banach spaces, every surjective Hausdorff isometry maps singleton sets onto singleton sets, and hence restricts to a surjective affine isometry between the underlying spaces. They also obtained the full setwise representation under additional geometric assumptions such as strict convexity or G\^ateaux smoothness.

Against this historical background, let $X$ be a real Banach space. We denote by $\cC(X)$ the family of all nonempty bounded closed convex subsets of $X$, endowed with the Hausdorff metric, and by $\cK(X)$ the subfamily of all nonempty compact convex subsets. Identifying $x\in X$ with the singleton $\{x\}$, the space $X$ is canonically embedded in both hyperspaces. The remaining problem for the full bounded closed convex hyperspace is therefore whether every surjective Hausdorff isometry
\[
F:\cC(X)\longrightarrow\cC(Y)
\]
between arbitrary real Banach spaces must be induced by a surjective affine isometry of the underlying spaces, without any additional smoothness or geometric assumptions.

In this paper, we give an affirmative answer. We prove that every surjective Hausdorff isometry $F:\cC(X)\to\cC(Y)$ maps singleton sets onto singleton sets and is of the form
\[
F(A)=T[A]\qquad(A\in\cC(X)),
\]
where $T:X\to Y$ is a surjective affine isometry. Thus, no Asplundness, smoothness, strict convexity, reflexivity, separability, or finite-dimensionality assumption is required.

The proof is metric in nature. Its main step is to show, by means of outer parallel bodies, that the hyperspace isometry preserves enough metric information to recover the inclusion order. Singleton preservation then follows from order minimality, and the classical Mazur--Ulam theorem yields the affine representation.

Our main result is the following.

\begin{mainthm}\label{thm:main}
\nonumber
Let $X$ and $Y$ be real Banach spaces, and let
\[
F:\cC(X)\longrightarrow\cC(Y)
\]
be a surjective isometry for the Hausdorff metric. Then $F$ maps singleton sets onto singleton sets. The map $T:X\to Y$ defined by
\[
F(\{x\})=\{T(x)\}\qquad(x\in X)
\]
is a surjective affine isometry, and
\begin{equation}\label{eq:mainrep}
F(A)=T[A]=\{T(x):x\in A\}\qquad(A\in\cC(X)).
\end{equation}
In particular,
\[
A\subseteq B\quad\Longleftrightarrow\quad F(A)\subseteq F(B)
\qquad(A,B\in\cC(X)),
\]
and $F(\cK(X))=\cK(Y)$.
\end{mainthm}

The paper is organized as follows. Section~2 recalls the representation of the support-function of the Hausdorff metric. Section~3 establishes the metric rigidity of outer parallel-body rays. Sections~4 and~5 show that these rays determine the inclusion order. Section~6 completes the proof of Theorem~\ref{thm:main} and records several consequences.

\section{Preliminaries}

Throughout the paper, all Banach spaces are real and nonzero. We write $X^*$ for the dual of $X$, and $B_X$ and $S_X$ for its closed unit ball and unit sphere. For a subset $M\subseteq X$, $\co M$ denotes its norm-closed convex hull. For $A,B\in\cC(X)$, the Hausdorff distance is
\[
\dH(A,B)=\max\left\{
\sup_{a\in A}\dist(a,B),\
\sup_{b\in B}\dist(b,A)
\right\}.
\]
We use the closed Minkowski sum
\[
A\oplus B:=\overline{A+B}
\]
and the usual scalar multiplication $tA$ for $t\ge0$, with $0A=\{0\}$.

For $A\in\cC(X)$, its support function is
\[
\sigma_A(x^*)=\sup_{a\in A}\pair{x^*}{a},\qquad x^*\in X^*.
\]
The support-function representation for bounded closed convex sets is discussed by Cheng et al. \cite{ChengEtAl2018} and Cheng and Zheng \cite{ChengZheng2021}. In particular,
\begin{equation}\label{eq:hausdorff-support}
\dH(A,B)=\sup_{x^*\in B_{X^*}}\abs{\sigma_A(x^*)-\sigma_B(x^*)}.
\end{equation}
Moreover, support functions determine these sets and their inclusion order:
\[
A\subseteq B\quad\Longleftrightarrow\quad \sigma_A\le\sigma_B\text{ on }X^*.
\]
These facts also follow directly from the separation of a point from a closed convex set.

By \eqref{eq:hausdorff-support}, the map $A\mapsto \sigma_A|_{B_{X^*}}$ is an isometric embedding into $\ell_\infty(B_{X^*})$. We continue to regard support functions as functions on all of $X^*$ and set
\[
\mathfrak{V}(X):=\{\sigma_A:A\in\cC(X)\},\qquad \mathfrak{E}(X):=\mathfrak{V}(X)-\mathfrak{V}(X).
\]
For $h\in \mathfrak{E}(X)$, define
\[
\norm{h}_\infty:=\sup_{x^*\in B_{X^*}}\abs{h(x^*)}.
\]
Every element of $\mathfrak{E}(X)$ is positively homogeneous and is therefore determined by its restriction to $B_{X^*}$. Thus, this formula defines a norm on $\mathfrak{E}(X)$. Since
\[
\sigma_{A\oplus B}=\sigma_A+\sigma_B,
\qquad
\sigma_{tA}=t\sigma_A\quad(t\ge0),
\]
$\mathfrak{V}(X)$ is a convex cone, and $\mathfrak{E}(X)$ is a real normed linear space.

The cone $\mathfrak{V}(X)$ is closed in $\mathfrak{E}(X)$ with respect to $\norm{\cdot}_\infty$. No completeness of $\mathfrak{E}(X)$ is needed.

For the unit ball, one has
\begin{equation}\label{eq:ball-support}
\sigma_{B_X}(x^*)=\norm{x^*}\qquad(x^*\in X^*),
\end{equation}
hence $\norm{\sigma_{B_X}}_\infty=1$. The outer parallel bodies satisfy
\begin{equation}\label{eq:parallel-support}
P_t^X(A)=A\oplus tB_X,
\qquad
\sigma_{P_t^X(A)}=\sigma_A+t\sigma_{B_X}
\qquad(t\ge0).
\end{equation}

The next lemma records the elementary lattice property needed in the following. For the related Banach-lattice construction, see Cheng et al. \cite{ChengEtAl2018}.

\begin{lemma}[Vector-lattice structure]\label{lem:lattice}
The space $\mathfrak{E}(X)$ is closed under pointwise maximum, pointwise minimum, and absolute value. Consequently, it is a vector lattice for pointwise order.
\end{lemma}

\begin{proof}
Write
\[
h_1=p_1-q_1,
\qquad
h_2=p_2-q_2,
\qquad
p_i,q_i\in \mathfrak{V}(X).
\]
Pointwise,
\[
h_1\vee h_2
=\bigl[(p_1+q_2)\vee(p_2+q_1)\bigr]-(q_1+q_2).
\]
The Sums of support functions belong to $\mathfrak{V}(X)$. Also, for $A,B\in\cC(X)$,
\[
\sigma_A\vee\sigma_B=\sigma_{\co(A\cup B)}\in \mathfrak{V}(X).
\]
Hence $h_1\vee h_2\in \mathfrak{E}(X)$. The identities
\[
h_1\wedge h_2=-\bigl((-h_1)\vee(-h_2)\bigr),
\qquad
\abs{h_1}=h_1\vee(-h_1)
\]
give the remaining assertions.
\end{proof}

\section{Metric rigidity of outer parallel bodies}

The following lemma uses the identity $\sigma_{B_X}=1$ on $S_{X^*}$.

\begin{lemma}[Unique interpolation on parallel-body rays]\label{lem:unique-interpolation}
Let $A\in\cC(X)$, and let $0\le s<r<t$. Then $P_r^X(A)$ is the unique $C\in\cC(X)$ satisfying
\[
\dH(P_s^X(A),C)=r-s,
\qquad
\dH(C,P_t^X(A))=t-r.
\]
\end{lemma}

\begin{proof}
Put $a=r-s$ and $b=t-r$. By \eqref{eq:parallel-support},
\[
\dH(P_s^X(A),P_r^X(A))=\norm{(r-s)\sigma_{B_X}}_\infty=r-s,
\]
and similarly $\dH(P_r^X(A),P_t^X(A))=t-r$.

Suppose that $C\in\cC(X)$ satisfies the prescribed equalities, and set
\[
h:=\sigma_C-\sigma_A-s\sigma_{B_X}.
\]
Equations~\eqref{eq:hausdorff-support} and \eqref{eq:parallel-support} give
\[
\norm{h}_\infty=a,
\qquad
\norm{(a+b)\sigma_{B_X}-h}_\infty=b.
\]
For every $x^*\in S_{X^*}$, these equalities imply
\[
-a\le h(x^*)\le a,
\qquad
a\le h(x^*)\le a+2b.
\]
Thus, $h(x^*)=a$ on $S_{X^*}$. Positive homogeneity yields $h=a\sigma_{B_X}$ on $X^*$. Therefore
\[
\sigma_C=\sigma_A+r\sigma_{B_X}=\sigma_{P_r^X(A)},
\]
 hence $C=P_r^X(A)$.
\end{proof}

The same calculation gives
\begin{equation}\label{eq:ray-distance}
\dH(P_s^X(A),P_t^X(A))=\abs{t-s}\qquad(s,t\ge0).
\end{equation}
Thus each map $t\mapsto P_t^X(A)$ is an isometric ray with unique metric interpolation.

\section{Rigidity of the image direction}

\begin{proposition}[Common image direction]\label{prop:common-direction}
There exists $\varphi\in \mathfrak{V}(Y)$ with $\norm{\varphi}_\infty=1$ such that
\begin{equation}\label{eq:image-ray-support}
\sigma_{F(P_t^X(A))}=\sigma_{F(A)}+t\varphi
\end{equation}
for every $A\in\cC(X)$ and $t\ge0$. Equivalently, there is $D\in\cC(Y)$ with $\varphi=\sigma_D$ such that
\begin{equation}\label{eq:image-ray-set}
F(P_t^X(A))=F(A)\oplus tD.
\end{equation}
\end{proposition}

\begin{proof}
Fix $A\in\cC(X)$ and define
\[
\eta_A(t):=\sigma_{F(P_t^X(A))}\in \mathfrak{V}(Y),
\qquad t\ge0.
\]
Because $F$ is surjective, every point of $\mathfrak{V}(Y)$ corresponds to a set with a preimage under $F$. Lemma~\ref{lem:unique-interpolation} therefore implies that, for $0\le s<r<t$, $\eta_A(r)$ is the unique point of $\mathfrak{V}(Y)$ at distances $r-s$ and $t-r$ from $\eta_A(s)$ and $\eta_A(t)$, respectively.

Set $\lambda=(r-s)/(t-s)$. Since $\mathfrak{V}(Y)$ is convex, the point
\[
(1-\lambda)\eta_A(s)+\lambda\eta_A(t)
\]
belongs to $\mathfrak{V}(Y)$. Its distances to the endpoints in $\mathfrak{E}(Y)$ are $r-s$ and $t-r$, by \eqref{eq:ray-distance}. Uniqueness gives
\begin{equation}\label{eq:affine-ray}
\eta_A(r)=(1-\lambda)\eta_A(s)+\lambda\eta_A(t).
\end{equation}
Let $\varphi_A=\eta_A(1)-\eta_A(0)$. Applying \eqref{eq:affine-ray} first to $(s,r,t)=(0,u,1)$ for $0<u<1$ and then to $(s,r,t)=(0,1,u)$ for $u>1$, we obtain
\[
\eta_A(u)=\eta_A(0)+u\varphi_A\qquad(u\ge0).
\]
Equation~\eqref{eq:ray-distance} gives $\norm{\varphi_A}_\infty=1$. Moreover,
\[
u^{-1}\eta_A(u)=u^{-1}\eta_A(0)+\varphi_A\in \mathfrak{V}(Y)\qquad(u>0).
\]
Since $\mathfrak{V}(Y)$ is closed in $\mathfrak{E}(Y)$, letting $u\to\infty$ yields $\varphi_A\in \mathfrak{V}(Y)$.

For $A,C\in\cC(X)$, equation~\eqref{eq:parallel-support} gives
\[
\dH(P_t^X(A),P_t^X(C))=\dH(A,C)\qquad(t\ge0).
\]
Applying $F$ and using the affine formulas for the image rays, we obtain
\[
\norm{\sigma_{F(A)}-\sigma_{F(C)}+t(\varphi_A-\varphi_C)}_\infty
=\dH(A,C)
\qquad(t\ge0).
\]
Dividing by $t$ and letting $t\to\infty$ gives $\varphi_A=\varphi_C$. Thus the direction is independent of $A$, proving \eqref{eq:image-ray-support}. Writing $\varphi=\sigma_D$ gives \eqref{eq:image-ray-set}.
\end{proof}

\begin{lemma}[Extremality of the common direction]\label{lem:extreme-direction}
Under the hypotheses of Proposition~\ref{prop:common-direction}, the common direction $\varphi$ is an extreme point of the unit ball $B_{\mathfrak{E}(Y)}$.
\end{lemma}

\begin{proof}
Suppose
\[
\varphi=\frac{\psi_1+\psi_2}{2},
\qquad
\psi_1,\psi_2\in B_{\mathfrak{E}(Y)}.
\]
Choose $C_1,D_1,C_2,D_2\in\cC(Y)$ such that
\[
\psi_1=\sigma_{C_1}-\sigma_{D_1},
\qquad
\psi_2=\sigma_{C_2}-\sigma_{D_2}.
\]
Set
\[
\sigma_H:=\sigma_{D_1}+\sigma_{D_2}
=\sigma_{D_1\oplus D_2}\in\mathfrak{V}(Y).
\]
Then
\[
\sigma_H+\psi_1
=\sigma_{C_1}+\sigma_{D_2}\in\mathfrak{V}(Y),
\qquad
\sigma_H+\psi_2
=\sigma_{D_1}+\sigma_{C_2}\in\mathfrak{V}(Y),
\]
and
\[
\sigma_H+2\varphi
=\sigma_{C_1}+\sigma_{C_2}\in\mathfrak{V}(Y).
\]
Moreover,
\[
2=\norm{2\varphi}_\infty
=\norm{\psi_1+\psi_2}_\infty
\le \norm{\psi_1}_\infty+\norm{\psi_2}_\infty
\le2.
\]
Hence $\norm{\psi_1}_\infty=\norm{\psi_2}_\infty=1$, so both $\sigma_H+\psi_1$ and $\sigma_H+\psi_2$ are metric midpoints between $\sigma_H$ and $\sigma_H+2\varphi$.

By surjectivity, choose $A\in\cC(X)$ such that $\sigma_H=\sigma_{F(A)}$. Proposition~\ref{prop:common-direction} yields
\[
\sigma_H+t\varphi=\sigma_{F(P_t^X(A))}\qquad(t\ge0).
\]
Transporting Lemma~\ref{lem:unique-interpolation} through $F$, we see that $\sigma_H+\varphi$ is the unique metric midpoint between $\sigma_H$ and $\sigma_H+2\varphi$ in $\mathfrak{V}(Y)$. Therefore
\[
\sigma_H+\psi_1=\sigma_H+\varphi=\sigma_H+\psi_2,
\]
and hence $\psi_1=\psi_2=\varphi$.
\end{proof}

\begin{lemma}[Identification of the common direction]\label{lem:identify-direction}
Under the hypotheses of Proposition~\ref{prop:common-direction},
\[
\varphi=\sigma_{B_Y}.
\]
\end{lemma}

\begin{proof}
Positive homogeneity and $\norm{\varphi}_\infty=1$ imply
\[
\abs{\varphi(y^*)}\le \sigma_{B_Y}(y^*)\qquad(y^*\in Y^*).
\]
By Lemma~\ref{lem:lattice},
\[
h:=\sigma_{B_Y}-\abs{\varphi}\in\mathfrak{E}(Y).
\]
For $y^*\in S_{Y^*}$, put $a=\varphi(y^*)\in[-1,1]$. Then $h(y^*)=1-\abs{a}$, and
\[
\abs{a+(1-\abs{a})}\le1,
\qquad
\abs{a-(1-\abs{a})}\le1.
\]
Consequently, $\norm{\varphi+h}_\infty\le1$ and $\norm{\varphi-h}_\infty\le1$. Since
\[
\varphi=\frac12\bigl[(\varphi+h)+(\varphi-h)\bigr],
\]
Lemma~\ref{lem:extreme-direction} implies $h=0$. Thus
\begin{equation}\label{eq:d-abs-one}
\abs{\varphi(y^*)}=1\qquad(y^*\in S_{Y^*}).
\end{equation}

Suppose first that $\dim Y\ge2$. The sphere $S_{Y^*}$ is path connected in the norm topology: non-antipodal points can be joined by normalizing their line segment, and antipodal points can be joined through a third direction. Since $\varphi=\sigma_D$ for some $D\in\cC(Y)$,
\[
\abs{\varphi(y_1^*)-\varphi(y_2^*)}
\le
\left(\sup_{y\in D}\norm{y}\right)\norm{y_1^*-y_2^*}.
\]
Hence $\varphi$ is norm-continuous. By \eqref{eq:d-abs-one}, its restriction to $S_{Y^*}$ takes only the values $-1$ and $1$, so it is constant. The constant cannot be $-1$, because
\[
\varphi(y^*)+\varphi(-y^*)
=\sup_{y\in D}\pair{y^*}{y}-\inf_{y\in D}\pair{y^*}{y}
\ge0.
\]
Thus $\varphi=1$ on $S_{Y^*}$, and positive homogeneity gives $\varphi=\sigma_{B_Y}$.

Now suppose $\dim Y=1$. Choose $e\in S_Y$ and $e^*\in S_{Y^*}$ with $e^*(e)=1$. Equation~\eqref{eq:d-abs-one} and the nonnegativity of the width give the possibilities
\[
\bigl(\varphi(e^*),\varphi(-e^*)\bigr)
\in\{(1,1),(1,-1),(-1,1)\}.
\]
These correspond, respectively, to $\varphi=\sigma_{B_Y}$, $D=\{e\}$, and $D=\{-e\}$. In either singleton case, the map
\[
Q_t(C):=C\oplus tD
\]
is a translation of $\cC(Y)$ and is therefore surjective. By \eqref{eq:image-ray-set}, however,
\[
Q_t=F\circ P_t^X\circ F^{-1}.
\]
For $t>0$, $P_t^X$ is not surjective. Indeed, for any $x^*\in S_{X^*}$,
\[
\sigma_{P_t^X(A)}(x^*)+\sigma_{P_t^X(A)}(-x^*)
=\sigma_A(x^*)+\sigma_A(-x^*)+2t
\ge2t,
\]
whereas every singleton has width zero. This contradicts the surjectivity of $Q_t$. The singleton cases are therefore impossible, and $\varphi=\sigma_{B_Y}$.
\end{proof}

\begin{theorem}[Preservation of outer parallel bodies]\label{thm:parallel-preservation}
Let $X$ and $Y$ be real Banach spaces, and let $F:\cC(X)\to\cC(Y)$ be a surjective Hausdorff isometry. Then
\begin{equation}\label{eq:parallel-preservation}
F(A\oplus tB_X)=F(A)\oplus tB_Y
\end{equation}
for every $A\in\cC(X)$ and $t\ge0$.
\end{theorem}

\begin{proof}
Proposition~\ref{prop:common-direction} and Lemma~\ref{lem:identify-direction} give $\varphi=\sigma_{B_Y}$. Substitution into \eqref{eq:image-ray-support} yields \eqref{eq:parallel-preservation}.
\end{proof}

\section{Recovering the inclusion order}

The next formula converts the Hausdorff distance into a one-sided comparison of support functions.

\begin{proposition}[Metric characterization of inclusion]\label{prop:metric-inclusion}
Let $A,B\in\cC(X)$. If $t>\dH(A,B)$, then
\begin{equation}\label{eq:one-sided-distance}
\dH(A\oplus tB_X,B)
=t+\sup_{x^*\in S_{X^*}}\bigl(\sigma_A(x^*)-\sigma_B(x^*)\bigr).
\end{equation}
Consequently, for every such $t$,
\begin{equation}\label{eq:metric-inclusion}
A\subseteq B
\quad\Longleftrightarrow\quad
\dH(A\oplus tB_X,B)\le t.
\end{equation}
\end{proposition}

\begin{proof}
Put $h=\sigma_A-\sigma_B$. Then $\norm{h}_\infty=\dH(A,B)<t$, and
\[
\dH(A\oplus tB_X,B)=\norm{h+t\sigma_{B_X}}_\infty.
\]
For $x^*\in S_{X^*}$, $\sigma_{B_X}(x^*)=1$ and $t+h(x^*)>0$. Positive homogeneity allows the supremum over $B_{X^*}$ to be taken over $S_{X^*}$, and positivity allows the absolute value to be omitted. Therefore
\[
\norm{h+t\sigma_{B_X}}_\infty
=\sup_{x^*\in S_{X^*}}\bigl(t+h(x^*)\bigr)
=t+\sup_{x^*\in S_{X^*}}h(x^*).
\]
This proves \eqref{eq:one-sided-distance}; no attainment of the supremum is required.

For bounded closed convex sets,
\[
A\subseteq B
\quad\Longleftrightarrow\quad
\sigma_A\le\sigma_B\text{ on }X^*
\quad\Longleftrightarrow\quad
\sup_{x^*\in S_{X^*}}h(x^*)\le0.
\]
Combining this with \eqref{eq:one-sided-distance} gives \eqref{eq:metric-inclusion}.
\end{proof}

\begin{theorem}[Order rigidity]\label{thm:order-rigidity}
Every surjective Hausdorff isometry $F:\cC(X)\to\cC(Y)$ is an order isomorphism:
\[
A\subseteq B
\quad\Longleftrightarrow\quad
F(A)\subseteq F(B)
\qquad(A,B\in\cC(X)).
\]
\end{theorem}

\begin{proof}
Fix $A,B\in\cC(X)$ and choose
\[
t>\dH(A,B)=\dH(F(A),F(B)).
\]
Proposition~\ref{prop:metric-inclusion} and Theorem~\ref{thm:parallel-preservation} give
\begin{align*}
A\subseteq B
&\Longleftrightarrow \dH(A\oplus tB_X,B)\le t\\
&\Longleftrightarrow \dH(F(A\oplus tB_X),F(B))\le t\\
&\Longleftrightarrow \dH(F(A)\oplus tB_Y,F(B))\le t\\
&\Longleftrightarrow F(A)\subseteq F(B).
\end{align*}
\end{proof}

\section{Proof of the main theorem and consequences}

\begin{proof}[Proof of the Main Theorem \ref{thm:main}]
By Theorem~\ref{thm:order-rigidity}, $F$ is an order isomorphism. The minimal elements of $\cC(X)$ are exactly the singleton sets: every singleton is minimal, whereas any set containing two distinct points properly contains a singleton. The same holds in $\cC(Y)$. Thus both $F$ and $F^{-1}$ preserve singleton sets.

Define $T:X\to Y$ by
\[
F(\{x\})=\{T(x)\}\qquad(x\in X).
\]
Singleton preservation in both directions implies that $T$ is surjective. For $x,z\in X$,
\begin{align*}
\norm{T(x)-T(z)}
&=\dH(\{T(x)\},\{T(z)\})\\
&=\dH(F(\{x\}),F(\{z\}))\\
&=\dH(\{x\},\{z\})
=\norm{x-z}.
\end{align*}
Hence $T$ is a surjective isometry, and the theorem of Mazur and Ulam \cite{MazurUlam1932} implies that $T$ is affine.

For $A\in\cC(X)$ and $x\in X$,
\[
x\in A
\quad\Longleftrightarrow\quad
\{x\}\subseteq A
\quad\Longleftrightarrow\quad
F(\{x\})\subseteq F(A)
\quad\Longleftrightarrow\quad
T(x)\in F(A).
\]
Since $T$ is onto, $F(A)=T[A]$. This proves \eqref{eq:mainrep}.
\end{proof}

\begin{corollary}[Preservation of compact convex sets]\label{cor:compact}
Under the hypotheses of Theorem~\ref{thm:main}, for every $A\in\cC(X)$,
\[
A\in\cK(X)
\quad\Longleftrightarrow\quad
F(A)\in\cK(Y).
\]
Consequently, $F(\cK(X))=\cK(Y)$.
\end{corollary}

\begin{proof}
The representation $F(A)=T[A]$ and the fact that $T$ is a homeomorphism show that compactness is preserved in both directions.
\end{proof}

\begin{corollary}[Preservation of closed convex hulls and line segments]\label{cor:hulls}
Under the hypotheses of Theorem~\ref{thm:main}, if $M\subseteq X$ is nonempty and bounded, then
\[
F(\co M)=\co T[M].
\]
In particular,
\[
F([x,z])=[T(x),T(z)]\qquad(x,z\in X).
\]
\end{corollary}

\begin{proof}
Since $T$ is affine, $T[\operatorname{co}M]=\operatorname{co}T[M]$. Since $T$ is a homeomorphism, it also preserves closures. Therefore
\[
F(\co M)=T[\co M]=\co T[M].
\]
The assertion about line segments follows by taking $M=\{x,z\}$.
\end{proof}

\subsection*{Acknowledgements}

The authors used Doubao-2.1 Turbo and ChatGPT-5.6 Plus during manuscript preparation to assist with literature search and citation organization. All mathematical arguments, results, and conclusions were independently developed and verified by the authors, who take full responsibility for the content of this article.

\subsection* {Funding}

This work was supported by the National Natural Science Foundation of China
(Grant No.~12271453).

%%%%%%%%%%% To ease editing, use normal size for the references:


\normalsize


\begin{thebibliography}{[HD82]}

%% Use the widest label as parameter above.
%% Reference items can be numbered or have labels of your choice, as below.
%% Arrange the items in the alphabetical order of surnames (and not in the order of labels).
%% For arXiv papers, give the version you are citing.
%% In IMPAN journals, only the title is italicized; boldface is not used.
%% Do NOT give the issue number unless the issues are paginated separately, as in Uspekhi below.
%% All reference items should be cited in the body of the article.
  
%% To ease editing, add:

\normalsize

%%%%%%%%%%%%%

\bibitem{Bandt1986}
C.~Bandt,
\emph{On the metric structure of hyperspaces with Hausdorff metric},
Math. Nachr. \textbf{129} (1986), 176--183.
DOI: 10.1002/mana.19861290113.

\bibitem{ChengDongZhang2013}
L.~Cheng, Y.~Dong, and W.~Zhang,
\emph{On stability of nonsurjective $\varepsilon$-isometries of Banach spaces},
J. Funct. Anal. \textbf{264} (2013), no.~3, 713--734.
DOI: 10.1016/j.jfa.2012.11.017.

\bibitem{ChengZhou2014}
L.~Cheng and Y.~Zhou,
\emph{On perturbed metric-preserved mappings and their stability characterizations},
J. Funct. Anal. \textbf{266} (2014), no.~8, 4995--5015.
DOI: 10.1016/j.jfa.2014.01.023.

\bibitem{ChengChengTuZhang2015}
L.~Cheng, Q.~Cheng, K.~Tu, and J.~Zhang,
\emph{A universal theorem for stability of $\varepsilon$-isometries on Banach spaces},
J. Funct. Anal. \textbf{269} (2015), no.~1, 199--214.
DOI: 10.1016/j.jfa.2014.09.013.

\bibitem{ChengTuZhang2016}
L.~Cheng, K.~Tu, and W.~Zhang,
\emph{On weak stability of $\varepsilon$-isometries on wedges and its applications},
J. Math. Anal. Appl. \textbf{433} (2016), no.~2, 1673--1689.
DOI: 10.1016/j.jmaa.2015.09.011.

\bibitem{ChengShenZhangZhou2017}
L.~Cheng, Q.~Shen, W.~Zhang, and Y.~Zhou,
\emph{More on stability of almost surjective $\varepsilon$-isometries of Banach spaces},
Sci. China Math. \textbf{60} (2017), no.~2, 277--284.
DOI: 10.1007/s11425-016-0288-0.

\bibitem{ChengEtAl2018}
L.~Cheng, Q.~Cheng, Q.~Shen, K.~Tu, and W.~Zhang,
\emph{A new approach to measures of noncompactness of Banach spaces},
Studia Math. \textbf{240} (2018), 21--45.
DOI: 10.4064/sm8448-2-2017.

\bibitem{ChengFangLuoSun2019}
L.~Cheng, Q.~Fang, S.~Luo, and L.~Sun,
\emph{On non-surjective coarse isometries between Banach spaces},
Quaest. Math. \textbf{42} (2019), no.~3, 347--362.
DOI: 10.2989/16073606.2018.1486034.

\bibitem{ChengDong2020}
L.~Cheng and Y.~Dong,
\emph{A note on the stability of nonsurjective $\varepsilon$-isometries of Banach spaces},
Proc. Amer. Math. Soc. \textbf{148} (2020), 4837--4844.
DOI: 10.1090/proc/15114.

\bibitem{ChengZheng2021}
L.~Cheng and Z.~Zheng,
\emph{A convex set-valued version of the Mazur--Ulam theorem on Asplund spaces},
J. Math. Anal. Appl. \textbf{498} (2021), no.~1, 124932.
DOI: 10.1016/j.jmaa.2021.124932.

\bibitem{ChengZheng2022}
L.~Cheng and Z.~Zheng,
\emph{A set-valued extension of the Mazur--Ulam theorem},
Studia Math. \textbf{263} (2022), 121--139.
DOI: 10.4064/sm200120-9-2.

\bibitem{Figiel1968}
T.~Figiel,
\emph{On non linear isometric embeddings of normed linear spaces},
Bull. Acad. Polon. Sci. S\'er. Sci. Math. Astronom. Phys. \textbf{16} (1968), 185--188.

\bibitem{Gevirtz1983}
J.~Gevirtz,
\emph{Stability of isometries on Banach spaces},
Proc. Amer. Math. Soc. \textbf{89} (1983), 633--636.
DOI: 10.1090/S0002-9939-1983-0718987-6.

\bibitem{GodefroyKalton2003}
G.~Godefroy and N.~J.~Kalton,
\emph{Lipschitz-free Banach spaces},
Studia Math. \textbf{159} (2003), 121--141.
DOI: 10.4064/sm159-1-6.

\bibitem{Gruber1978}
P.~M.~Gruber,
\emph{Stability of isometries},
Trans. Amer. Math. Soc. \textbf{245} (1978), 263--277.
DOI: 10.1090/S0002-9947-1978-0511409-2.

\bibitem{GruberLettl1980}
P.~M.~Gruber and G.~Lettl,
\emph{Isometries of the space of convex bodies in Euclidean space},
Bull. London Math. Soc. \textbf{12} (1980), 455--462.
DOI: 10.1112/blms/12.6.455.

\bibitem{HyersUlam1945}
D.~H.~Hyers and S.~M.~Ulam,
\emph{On approximate isometries},
Bull. Amer. Math. Soc. \textbf{51} (1945), 288--289.
DOI: 10.1090/S0002-9904-1945-08337-2.

\bibitem{HyersUlam1947}
D.~H.~Hyers and S.~M.~Ulam,
\emph{On approximate isometries on the space of continuous functions},
Ann. of Math. \textbf{48} (1947), 285--289.
DOI: 10.2307/1969192.

\bibitem{LindenstraussSzankowski1985}
J.~Lindenstrauss and A.~Szankowski,
\emph{Nonlinear perturbations of isometries},
in: Colloquium in Honor of Laurent Schwartz, Vol.~1 (Palaiseau, 1983),
Ast\'erisque \textbf{131} (1985), 357--371.

\bibitem{MazurUlam1932}
S.~Mazur and S.~Ulam,
\emph{Sur les transformations isom\'etriques d'espaces vectoriels norm\'es},
C. R. Acad. Sci. Paris \textbf{194} (1932), 946--948.

\bibitem{OmladicSemrl1995}
M.~Omladi\v{c} and P.~\v{S}emrl,
\emph{On non linear perturbations of isometries},
Math. Ann. \textbf{303} (1995), 617--628.
DOI: 10.1007/BF01461008.

\bibitem{Qian1995}
S.~Qian,
\emph{$\varepsilon$-isometric embeddings},
Proc. Amer. Math. Soc. \textbf{123} (1995), 1797--1803.
DOI: 10.1090/S0002-9939-1995-1246532-1.

\bibitem{Radstrom1952}
H.~R\aa dstr\"om,
\emph{An embedding theorem for spaces of convex sets},
Proc. Amer. Math. Soc. \textbf{3} (1952), 165--169.
DOI: 10.1090/S0002-9939-1952-0045938-2.

\bibitem{Schneider1975}
R.~Schneider,
\emph{Isometrien des Raumes der konvexen K\"orper},
Colloq. Math. \textbf{33} (1975), 219--224.
DOI: 10.4064/cm-33-2-219-224.

\bibitem{SemrlVaisala2003}
P.~\v{S}emrl and J.~V\"ais\"al\"a,
\emph{Nonsurjective nearisometries of Banach spaces},
J. Funct. Anal. \textbf{198} (2003), 268--278.
DOI: 10.1016/S0022-1236(02)00065-2.

\bibitem{Vestfrid2015}
I.~A.~Vestfrid,
\emph{Stability of almost surjective $\varepsilon$-isometries of Banach spaces},
J. Funct. Anal. \textbf{269} (2015), 2165--2170.
DOI: 10.1016/j.jfa.2015.04.009.

\bibitem{ZhouZhangLiu2018}
Y.~Zhou, Z.~Zhang, and C.~Liu,
\emph{On representation of isometric embeddings between Hausdorff metric spaces of compact convex subsets},
Houston J. Math. \textbf{44} (2018), no.~3, 917--925.

\bibitem{ZhouZhangLiu2021}
Y.~Zhou, Z.~Zhang, and C.~Liu,
\emph{Representation of surjective additive isometric embeddings between Hausdorff metric spaces of compact convex subsets in finite-dimensional Banach spaces},
Studia Math. \textbf{257} (2021), no.~1, 111--119.
DOI: 10.4064/sm200326-9-6.

\bibitem{Zhou2021}
Y.~Zhou,
\emph{Representation of surjective additive isometric embeddings between Hausdorff metric spaces of certain bounded closed convex subsets of Banach spaces},
Houston J. Math. \textbf{47} (2021), no.~1, 97--114.

\end{thebibliography}
\end{document}